\documentclass[10pt]{article}
\usepackage{amssymb,amsmath,amsfonts,   amsthm,nomencl,mathrsfs } 
\usepackage[arrow, matrix, curve]{xy}
\usepackage[latin1]{inputenc}
\usepackage{a4wide}
\usepackage{color}

\newcommand{\IC}{\mathbb{C}}
\newcommand{\IR}{\mathbb{R}}

\newcommand{\question}[1]{\leavevmode{\marginpar{\tiny%
$\hbox to 0mm{\hspace*{-0.5mm}$\leftarrow$\hss}%
\vcenter{\vrule depth 0.1mm height 0.1mm width \the\marginparwidth}%
\hbox to 0mm{\hss$\rightarrow$\hspace*{-0.5mm}}$\\\relax\raggedright #1}}}

\newcommand{\IN}{\mathbb{N}}
\newcommand{\IZ}{\mathbb{Z}}

\newcommand{\pa}{\slash\slash}

\newcommand{\trasp}{/\!/}

\theoremstyle{plain}            
\newtheorem{theorem}{theorem}[section]
\newtheorem{Lemma}[theorem]{Lemma}
\newtheorem{Corollary}[theorem]{Corollary}
\newtheorem{Theorem}[theorem]{Theorem}
\newtheorem{Proposition}[theorem]{Proposition}

\theoremstyle{definition}       
\newtheorem{Definition}[theorem]{Definition}

\newtheorem{Remark}[theorem]{Remark}

\newtheorem{Notation}[theorem]{Notation}

\allowdisplaybreaks[1]

\begin{document}

\begin{titlepage}
\vspace{.5cm}

{\Large \bf Odd characteristic classes in entire cyclic homology and equivariant loop space homology}\\
\begin{center}
\vskip 10mm
{\large Sergio Cacciatori$^{1}$ and Batu G\"uneysu$^{2}$} \\
\end{center}
{\small
$^1$ Dipartimento di Scienza e Alta Tecnologia, Universit\`a dell'Insubria, via Valleggio 11,
22100, Como, Italy and INFN - Sezione di Milano, via Celoria 16, 20133, Milano, Italy \\
{\tt email: sergio.cacciatori@uninsubria.it}

\

$^2$ Mathematisches Institut, Universität Bonn, Endenicher Allee 60, 53115 Bonn, Germany\\ {\tt email: gueneysu@math.uni-bonn.de}
}
\vskip 10mm
\begin{center}
{\bf Abstract}
\end{center}
{\small Given a compact manifold $M$ and a smooth map $g:M\to U(l\times l;\IC)$ from $M$ to the Lie group of unitary $l\times l$ matrices with entries in $\IC$, we construct a Chern character $\mathrm{Ch}^-(g)$ which lives in the odd part of the equivariant (entire) cyclic Chen-normalized cyclic complex $\mathscr{N}_{\epsilon}(\Omega_{\mathbb{T}}(M\times \mathbb{T}))$ of $M$, and which is mapped to the odd Bismut-Chern character under the equivariant Chen integral map. It is also shown that the assignment $g\mapsto  \mathrm{Ch}^-(g)$ induces a well-defined group homomorphism from the $K^{-1}$ theory of $M$ to the odd homology group of $\mathscr{N}_{\epsilon}(\Omega_{\mathbb{T}}(M\times \mathbb{T}))$. }
 \vspace{2mm} \vfill \hrule width 3.cm





\end{titlepage}



Let $M$ be a closed Riemannian spin manifold with its Clifford multiplication 
$$
c:\Omega(M)\longrightarrow \mathrm{End}(S)
$$
and its Dirac operator $D$ acting in $L^2(M,S)$, and given $g\in C^{\infty}(M,U(l\times l;\IC))$ let $D_g$ denote the twisted Dirac operator
$$
D_g:= g^{-1}Dg = D+ c(g^{-1}d g),
$$
considered to be acting on $L^2(M,S\otimes \IC^l)$. Then with 
$$
D_{g,s}:=(1-s)D+s D_g, \quad s\in [0,1],
$$
the odd dimensional variant of Atiyah-Singer's 'index' theorem states that if $M$ is odd dimensional, then \cite{getzler} 
\begin{align}\label{psnyny}
 \frac{1}{2\pi }\int^1_0 \mathrm{Tr} \left[\dot{D}_{g,s} \exp\left(-  D_{g,s}^2\right)\right] d s = \int_M \hat{A}(M)\wedge \mathrm{ch}^-(g),
\end{align}
where $\mathrm{ch}^-(g)\in \Omega^-(M)$ denotes the odd Chern character. The left hand side of (\ref{psnyny}) is precisely the spectral flow $\mathrm{sf}(D,D_g)$ \cite{getzler}. Furthermore, on the RHS of this formula, the odd Chern character can be obtained integration along the fiber of $M\times I \to M$ of the even Chern character of an appropriately chosen connection on $M\times I$ \cite{getzler}. In fact, this formula can be proved by noting the LHS admits an infinite dimensional version of such an even/odd periodicity \cite{bf1,bf2} in terms of the eta form.\\
Being motivated by the considerations of Atiyah and Bismut \cite{atiyah,bismut} for the even-dimensional case, one finds that another very elegant and geometric, however purely formal, way to prove (\ref{psnyny}) is to assume the existence of a Duistermaat-Heckmann localization formula for the smooth loop space $LM$: indeed, the spin structure on $M$ induces an orientation on $LM$ \cite{atiyah} and the path integral formalism entails the elegant, however mathematically ill-defined, formula (the even-dimensional variant of this formula is well-known \cite{bismut} and the odd-dimensional case can be proved similarly \cite{shu})
\begin{align}\label{shus}
 \frac{1}{2\pi }\int^1_0 \mathrm{Tr} \left[\dot{D}_{g,s} \exp\left(-  D_{g,s}^2\right)\right] d s= \int_{LM} \exp\left( -\beta\right)\wedge  \mathrm{Bch}^-(g),
\end{align}
where $\beta=\beta_0+\beta_2 \in \Omega^+(LM)$ denotes the even differential form defined on smooth vector fields $X,Y$ on $LM$ by
$$
\beta_0(X):=\int^1_0  | X_s|^2  d s ,\quad \beta_2(X,Y):= \int^1_0  \left( \nabla X_s/ \nabla s ,Y_s\right)  d s,
$$
and where $\mathrm{Bch}^-(g)\in \Omega^-(M)$ denotes the odd Bismut-Chern character \cite{bismut122, wilson}. Now both differential forms $\exp( -\beta)$ and $ \mathrm{Bch}^-(g)$ are equivariantly closed (cf. Section \ref{posmmaa2} for the definition of the degree $-1$ differential $P$),
$$
(d+P)\exp( -\beta)=0=(d+P) \mathrm{Bch}^-(g)
$$
and so is their product. As the fixed point set of the $\mathbb{T}$-action on $LM$ given by rotating every loop is precisely $M\subset LM$ , a hypothetical Duistermaat-Heckmann localization formula immediately gives
$$
\int_{LM}\exp( -\beta)\wedge  \mathrm{Bch}^-(g) = \int_{M}  \hat{A}(M)\wedge \exp(-\beta)|_M\wedge   \mathrm{Bch}^-(g)|_{M}, 
$$
as $\hat{A}(M)$ is the inverse of the (appropriately renormalized) Euler class of the normal bundle of $M\subset LM$. This proves (\ref{psnyny}), as clearly $\exp(-\beta)|_M=1$ and by construction $\mathrm{Bch}^-(g)|_{M}= \mathrm{ch}^-(g)$.\\
A direct implementation of the above arguments is not possible, as the right hand side of formula (\ref{shus}) is not well-defined for various reasons. For example, there exists no volume measure on $LM$, while smooth loops have Wiener measure zero, and, on the other hand, it is notoriously difficult to produce a variant of the super complex $(\Omega(LM), d+P)$ if one replaces $LM$ with the smooth Banach manifold of \emph{continuous loops}. Nevertheless and strikingly, the above formal manipulations lead to the powerful machinery of hypoelliptic Dirac and Laplace operators, as is explained in \cite{bismut122} and the references therein.\\
However, a possible way out of these problems has been proposed by Getzler, Jones and Petrack (GJP) \cite{gjp} \cite{getzler2}. In this approach, the idea is to take as model for $\Omega(LM)$ the space of equivariant Chen integrals: these are given by the image of a morphism of super complexes (cf. Section \ref{posmmaa2} below for the relevant definitions)
$$
\rho:\big( \mathscr{N}_{\epsilon}(\Omega_{\mathbb{T}}(M\times \mathbb{T})), b+B\big)\longrightarrow
\big(\widehat{\Omega}(LM), d+P\big).
$$
Above, $\mathscr{N}_{\epsilon}(\Omega_{\mathbb{T}}(M\times \mathbb{T}))$ denotes the Chen-normalized entire cyclic (or Connes) complex of the locally convex unital DGA $\Omega_{\mathbb{T}}(M\times \mathbb{T})$, and $\widehat{\Omega}(LM)$ denotes a completed space of smooth differential forms on $LM$. Now the GJP-program for infinite dimensional localization is as follows: here it is conjectured that the composition
$$
\int_{LM} \exp( -\beta)\wedge \rho(\cdot):  \mathscr{N}_{\epsilon}(\Omega_{\mathbb{T}}(M\times \mathbb{T}))\longrightarrow \IC,
$$
\emph{is} a mathematically well-defined continuous functional, and that
\begin{itemize}
\item $\int_{LM} \exp( -\beta)\wedge \rho(\cdot)$ is odd (as $LM$ is formally odd-dimensional if $M$ is so \cite{bismut122}) and co-closed, meaning that it vanishes on the exact elements of $\mathscr{N}_{\epsilon}(\Omega_{\mathbb{T}}(M\times \mathbb{T}))$ ,
\item if $w\in \mathscr{N}_{\epsilon}(\Omega_{\mathbb{T}}(M\times \mathbb{T}))$ is closed, then one has the 'Duistermaat-Heckmann localization formula'
$$
\int_{LM} \exp( -\beta)\wedge \rho(w)  =\int_M \hat{A}(TM) \wedge \rho(w)|_M.
$$

\end{itemize}

If in addition one could canonically construct an element 
$$
\mathrm{Ch}^-(g)\in\mathscr{N}^-_{\epsilon}(\Omega_{\mathbb{T}}(M\times \mathbb{T}))
$$
 such that 
\begin{itemize}
\item[i)] $\mathrm{Ch}^-(g)$ is closed
\item[ii)] $\rho(\mathrm{Ch}^-(g))=\mathrm{Bch}^-(g)$
\item[iii)] $ \rho(\mathrm{Ch}^-(g))|_M=\mathrm{ch}^-(g)$,
\end{itemize}
then from the above observations we would immediately obtain a proof of (\ref{psnyny}) within the GJP-program for infinite dimensional localization. Note that in the even dimensional case such a Chern character has been constructed as an even cycle in $\mathscr{N}_{\epsilon}(\Omega_{\mathbb{T}}(M\times \mathbb{T}))$ in \cite{gjp}.\vspace{2mm}

The aim of this paper is precisely to construct a canonically given element 
$$
\mathrm{Ch}^-(g)\in\mathscr{N}^-_{\epsilon}(\Omega_{\mathbb{T}}(M\times \mathbb{T}))
$$
satisfying the above properties i), ii), iii). In fact, our main results Theorem \ref{main} and Theorem \ref{main2} below construct $\mathrm{Ch}^-(g)$ for $M$ a compact manifold (possibly with boundary), which satisfies i) and iii) and in addition ii) if $M$ is closed (so that $LM$ is a well-defined smooth Fr\'{e}chet manifold). We also show in Theorem \ref{main} that the assignment $g\mapsto \mathrm{Ch}^-(g)$ induces a well-defined group homomorphism 
$$
\mathsf{K}^{-1}(M) \longrightarrow  \mathscr{N}(\Omega_{\mathbb{T}}(M\times \mathbb{T})).
$$
Finally, taking for granted that the even variant of $\mathrm{Ch}^-(g)$ and $\mathrm{BCh}^-(g)$ have been previously defined \cite{gjp,bismut}, we establish an even/odd periodicity, relating these constructions to ours, showing another analogy to (\ref{psnyny}).\vspace{5mm}

{\bf Acknowledgements:} The authors would like to thank Jean-Michel Bismut, Markus Pflaum and Shu Shen for their discussions. We are very grateful to Matthias Ludewig for sharing his construction of the equivariant Chen integral map with us.

\section{Cyclic bar complex of a differential graded algebra (DGA)}

In the sequel, we understand all our linear spaces to be over $\IC$. Assume we are given a unital DGA $\Omega$, that is,
\begin{itemize}
	\item $\Omega$ is a unital algebra 
	\item $\Omega=\bigoplus^{\infty}_{j=-\infty} \Omega^j$ is graded into subspaces $\Omega^j\subset \Omega$ such that $\Omega^i\Omega^j\subset \Omega^{i+j}$ for all $i,j\in\IZ$, there is a degree $+1$ differential $d:\Omega\to \Omega$ which satisfies the graded Leibnitz rule.
\end{itemize}

Note that the space $\underline{\Omega}:=\Omega/(\IC\cdot \mathbf{1})$ is a graded linear space (but not canonically an algebra), and the space of cyclic chains $\mathscr{C}(\Omega)$ is defined as
$$
\mathscr{C}(\Omega):= \bigoplus^{\infty}_{n=0}\Omega \otimes  \underline{\Omega}^{ \otimes  n}.
$$
We give $\Omega \otimes  \underline{\Omega}^{ \otimes  n}$ the grading
$$
\Omega \otimes  \underline{\Omega}^{ \otimes  n}=\bigoplus^{\infty}_{j=0} \bigoplus_{j_0+\cdots+ j_n=j-n}\Omega^{j_0} \otimes \underline{\Omega}^{j_1}\otimes\cdots \otimes \underline{\Omega}^{j_n},
$$
which induces a linear map
$$
\Gamma:\mathscr{C}(\Omega)\longrightarrow \mathscr{C}(\Omega),\quad \Gamma (w_0,w_1,\dots):=\big((-1)^{\mathrm{deg}(w_0)}w_0,(-1)^{\mathrm{deg}(w_1)}w_1,\dots\big).
$$
Since we have $\Gamma^2=1$, we can define a superstructure $\mathscr{C}(\Omega)=\mathscr{C}^{+}(\Omega)\oplus \mathscr{C}^{-}(\Omega)$ by setting
$$
\mathscr{C}^{\pm}(\Omega):= \{w\in \mathscr{C}(\Omega): \Gamma w= \pm w\}.
$$

The following notation will be useful in the sequel:

\begin{Notation} Given $a\in \Omega \otimes  \underline{\Omega}^{ \otimes  n}$ we define  
$$
\left\langle a\right\rangle:=(\dots, a,\dots)\in \mathscr{C}(\Omega)
$$
to be the cochain which has $a$ in its $n$-th slot and 0 anywhere else.
\end{Notation}

We have the Hochschild map of the DGA-category 
$$
b:\mathscr{C}(\Omega)\longrightarrow \mathscr{C}(\Omega)
$$
defined on $\Omega^{j_0} \otimes \underline{\Omega}^{j_1}\otimes\cdots \otimes \underline{\Omega}^{j_n}$ by
{
\begin{align*}
 b  \left\langle  \omega_0 \otimes \cdots \otimes \omega_n \right\rangle=& \left\langle d \omega_0\otimes \cdots \otimes \omega_i \otimes \cdots \otimes \omega_n\right\rangle\\
& -\sum_{i=1}^n (-1)^{j_0+\ldots +j_{i-1} -i+1}\left\langle  \omega_0\otimes \cdots \otimes d\omega_i \otimes \cdots \otimes \omega_n\right\rangle \\
& -\sum_{i=0}^{n-1} (-1)^{j_0+\ldots +j_i-i} \left\langle \omega_0\otimes \cdots \otimes \omega_i \omega_{i+1} \otimes \cdots \otimes \omega_n \right\rangle\\
& +(-1)^{(j_n-1)(j_0+\ldots +j_{n-1}-n+1)} \left\langle (\omega_n \omega_0) \otimes \omega_1 \otimes \cdots \otimes \omega_{n-1}\right\rangle,
\end{align*}
}
and Connes' operator 
$$
B:\mathscr{C}(\Omega)\longrightarrow \mathscr{C}(\Omega),
$$
which is defined on $\Omega^{j_0} \otimes \underline{\Omega}^{j_1}\otimes\cdots \otimes \underline{\Omega}^{j_n}$ by 
\begin{align*}
 B \left\langle \omega_0 \otimes \cdots \otimes \omega_n \right\rangle=\sum_{i=0}^n (-1)^{(r_{i-1}+1)(r_n-r_{i-1})} \left\langle 1\otimes \omega_i \otimes \cdots \otimes \omega_n \otimes \omega_0 \otimes \cdots \otimes \omega_{i-1}\right\rangle,
\end{align*}
with $r_l=j_0+\cdots +j_l-l$. It is a well-known fact that one has
\begin{align*}
b^2=0,\quad B^2=0, \quad bB+bB=0,\quad \Gamma   b=-\Gamma b, \quad \Gamma B=-\Gamma B.
\end{align*}

We get the super complex
\begin{align}\label{cycl}
\mathscr{C}^{+}(\Omega)\xrightarrow{b+B}\mathscr{C}^{-}(\Omega)\xrightarrow{b+B}\mathscr{C}^{+}(\Omega). 
\end{align}

%

The subspace $\mathscr{D}(\Omega)\subset \mathscr{C}(\Omega)$ is defined to be the linear span of all $w\in \mathscr{C}(\Omega)$ that satisfy one of the following relations:\\
$\bullet$ for all $n\in\IN$ there exists $1\leq r\leq n$, $f\in \Omega^0$, $\omega_0\in \Omega$, { $\omega_s\in \underline{\Omega}$, $s\neq r$,} with
\begin{align} \label{titty0}
\left\langle w_n\right\rangle=\left\langle \omega_{0}\otimes \cdots\otimes \omega_{r-1}\otimes f\otimes\omega_{r+1}\otimes \cdots\otimes \omega_{n}\right\rangle.
\end{align}
$\bullet$ for all $n\in\IN$ there exists $1\leq r\leq n$, $f\in \Omega^0$, $\omega_0\in \Omega$, { $\omega_s\in \underline{\Omega}$, $s\neq r$,} with
\begin{align}\label{titty}
&\left\langle \omega_{0}\otimes \cdots\otimes \omega_{r-1} f\otimes\omega_{r+1}\otimes \cdots\otimes \omega_{n}\right\rangle+\left\langle \omega_{0}\otimes \cdots\otimes \omega_{r-1}\otimes df\otimes\omega_{r+1}\otimes \cdots\otimes \omega_{n}\right\rangle\cr
&\quad-\left\langle \omega_{0}\otimes \cdots\otimes \omega_{r-1}\otimes f\omega_{r+1}\otimes \cdots\otimes \omega_{n}\right\rangle.
\end{align}
The maps $\Gamma,b,B$ map $\mathscr{D}(\Omega)$ to itself, so that with
$$
 \mathscr{D}^{\pm}(\Omega):= \{w\in \mathscr{D}(\Omega): \Gamma w= \pm w\},
$$
there is a super complex
$$
 \mathscr{D}^{+}(\Omega)\xrightarrow{b+B}\mathscr{D}^{-}(\Omega)\xrightarrow{b+B}\mathscr{D}^{+}(\Omega).
$$
With $\mathscr{N}^{\pm}(\Omega):= \mathscr{C}^{\pm}(\Omega) / \mathscr{D}^{\pm}(\Omega)$, the induced quotient complex 
$$
  \mathscr{N}^{+}(\Omega)\xrightarrow{b+B}\mathscr{N}^{-}(\Omega)\xrightarrow{b+B}\mathscr{N}^{+}(\Omega). 
$$


Whenever there is no danger of confusion, the equivalence class of $w\in  \mathscr{C}(\Omega) $ in $\mathscr{N}(\Omega)$ is denoted with the same symbol again.

\section{Entire cyclic homology of a locally convex unital DGA}

We recall that a topological vector space is called locally convex, if the topology is induced by a family of seminorms, noting that then the topology is equivalent to the topology induced by all continuous seminorms.

\begin{Definition} By a locally convex unital DGA we understand a unital DGA $\Omega$ which is also a locally convex Hausdorff space, such that
\begin{itemize}

\item the differential is continuous, e.g., for every continuous seminorm $\varepsilon$ on $\Omega$ there exists a continuous seminorm $\varepsilon'$ on $\Omega$ such that 
\begin{align}\label{conti}
\epsilon( d\omega )\leq \epsilon'(\omega)\quad\text{ for all $\omega \in \Omega$}
\end{align}
\item the multiplication is jointly continuous, e.g., for every continuous seminorm $\varepsilon$ on $\Omega$ there exists a continuous seminorm $\varepsilon'$ on $\Omega$ such that 
\begin{align}\label{conti2}
\varepsilon( \omega_1\omega_2 )\leq\varepsilon'(\omega_1)\varepsilon'(\omega_2)\quad\text{ for all $\omega_1,\omega_2 \in \Omega$.}
\end{align}

\end{itemize}
\end{Definition}

The space $\underline{\Omega}$ becomes a graded locally convex Hausdorff space, and we equip the algebraic tensor product $\Omega\otimes \underline{\Omega}^{\otimes n}$ with the induced family of $\pi$-tensor seminorms, that is, 
$$
\varepsilon_n(\omega)= \inf \left\{\sum_{\alpha}\varepsilon(\omega_0^{(1)}) \cdots \varepsilon(\omega_n^{(\alpha)}): \omega=\sum_{\alpha}\omega_0^{(\alpha)}\otimes\cdots\otimes\omega_n^{(\alpha)}\right\},
$$
where the sum runs through all representations of $\omega$ as a finite sum of elementary tensors, and where $\epsilon$ is a continuous seminorm on $\Omega$.

\begin{Definition} The space of \emph{entire cyclic chains} $\mathscr{C}_{\epsilon}(\Omega)$ is defined to be the closure of $\mathscr{C}(\Omega)$ with respect to the seminorms 
$$
\kappa_{\varepsilon}(w):=\sum^{\infty}_{n=0} \frac{\varepsilon_n(w_n)}{ \sqrt{n!}},
$$
where $\varepsilon$ is an arbitrary continuous seminorm on $\Omega$.
\end{Definition}

The space $\mathscr{C}_{\epsilon}(\Omega)$ is a complete locally convex Hausdorff space. Note that the above family of seminorms is equivalent to the familiy of seminorms 
$$
\kappa_{\varepsilon,l}(w):=\sum^{\infty}_{n=0} \frac{\varepsilon_n(w_n)l^n}{ \sqrt{n!}} <\infty,
$$
where $\varepsilon$ is an arbitrary continuous seminorm on $\Omega$ and $l\in\IN$, as $l\varepsilon$ is again a continuous seminorm and the $\varepsilon_n$'s are cross semi-norms. Thus, our growth conditions are modelled on the entire growth conditions for ungraded Banach algebras by Getzler/Szenes from \cite{gs}. \\
Before stating the next auxiliary result, we recall that a continous linar map from a locally convex Hausdorff space $\mathscr{X}$ to a complete locally convex Hausdorff space $\mathscr{Y}$ can be uniquely extended to a continuous linear map $\hat{\mathscr{X}}\to \mathscr{Y}$, noting that the completion $\hat{\mathscr{X}}$ is Hausdorff again. This can be proved precisely as for normed spaces.

\begin{Lemma} The operators $\Gamma, b,B$ map $\mathscr{C}(\Omega)$ continuously to itself, in particular, with
$$
 \mathscr{C}^{\pm}_{\epsilon}(\Omega):= \{w\in \mathscr{C}_{\epsilon}(\Omega): \Gamma w= \pm w\},
$$
there is a well-defined super complex
\begin{align}\label{entcycl}
 \mathscr{C}^{+}_{\epsilon}(\Omega)\xrightarrow{b+B}\mathscr{C}^{-}_{\epsilon}(\Omega)\xrightarrow{b+B}\mathscr{C}^{+}_{\epsilon}(\Omega).
\end{align}
\end{Lemma}

\begin{proof} Let $\varepsilon$ be an arbitrary continuous seminorm on $\Omega$. Clearly, one has $\kappa_{\varepsilon}( \Gamma w)\leq  \kappa_{\varepsilon}( w)$ for all $w\in \mathscr{C}(\Omega)$. \\ 
Pick continuous seminorms $\varepsilon',\varepsilon''$ on $\Omega$ such that for all $\omega \in \Omega$ one has $\varepsilon(d\omega)\leq  \varepsilon''(\omega)$ and such that for all $\omega_1,\omega_2 \in \Omega$ one has $\varepsilon( \omega_1\omega_2)\leq \varepsilon'(\omega_1)\varepsilon'(\omega_2)$. Using $n+1\leq 2^n$ it is then easily checked that
\begin{align*}
\kappa_{\varepsilon}( b w)\leq   C\max(\kappa_{\varepsilon'},\kappa_{\varepsilon''})( w)\quad\text{ for all $w\in \mathscr{C}(\Omega)$. } 
\end{align*}
Likewise, it follows immediately that $\kappa_{\varepsilon} (B w)\leq  C\kappa_{\varepsilon}(w)$ for all $w\in \mathscr{C}(\Omega)$. 
\end{proof}

Defining the subspace $\mathscr{D}_{\epsilon}(\Omega)\subset \mathscr{C}_{\epsilon}(\Omega)$ as the closure of $\mathscr{D}(\Omega)$, it follows automatically that the maps
$\Gamma, b,B$ map $\mathscr{D}(\Omega)$ continuously to itself, too, producing with 
$$
\mathscr{N}^{\pm}_{\epsilon}(\Omega):= \mathscr{C}^{\pm}_{\epsilon}(\Omega) / \mathscr{D}^{\pm}_{\epsilon}(\Omega) 
$$
the quotient complex
\begin{align}\label{ass}
  \mathscr{N}^{+}_{\epsilon}(\Omega)\xrightarrow{b+B}\mathscr{N}^{-}_{\epsilon}(\Omega)\xrightarrow{b+B}\mathscr{N}^{+}_{\epsilon}(\Omega) .
\end{align}

Finally we can give:

\begin{Definition} The complex (\ref{entcycl}) is called the (reduced) \emph{entire cyclic complex} of $\Omega$ and its homology groups are denoted with $\mathsf{HC}^{\pm}_{\epsilon}(\Omega)$. Likewise, the complex (\ref{ass}) is called the (reduced) \emph{Chen-normalized entire cyclic complex} of $\Omega$ and its homology groups are denoted with $\mathsf{HN}^{\pm}_{\epsilon}(\Omega)$.
\end{Definition}

Above, 'reduced' refers to the fact that we work with $\Omega\otimes \underline{\Omega}^{\otimes n}$ rather than $\Omega^{\otimes (n+1)}$, which leads to a simpler formula for the Connes differential $B$.

\section{The unital locally convex DGA $\Omega_{\mathbb{T}}(N\times \mathbb{T})$}\label{posmmaa}

Assume $N$ is a manifold (possibly with boundary) and denote with $\mathbb{T}$ the $1$-sphere. We denote by $\Omega_{\mathbb{T}}(N\times \mathbb{T})$ the smooth $\mathbb{T}$-invariant differential forms on $N\times \mathbb{T}$, where $\mathbb{T}$ acts trivially on $N$ and by rotation on itself. Every element of  $\Omega_{\mathbb{T}}(N\times \mathbb{T})$ can be uniquely written in the form {$\alpha+\vartheta_\mathbb{T}\wedge\beta$} for some $\alpha,\beta\in \Omega(N)$, where $ \vartheta_\mathbb{T} $ denotes the canonical $1$-form on $\mathbb{T}$. We turn $\Omega_{\mathbb{T}}(N\times \mathbb{T})$ into a unital algebra by means of $\Omega_{\mathbb{T}}(N\times \mathbb{T})\subset  \Omega(N\times \mathbb{T})$, and give $\Omega_{\mathbb{T}}(N\times \mathbb{T})$ the grading
$$
{\alpha+\vartheta_\mathbb{T} \wedge \beta}\in \Omega^j_{\mathbb{T}}(N\times \mathbb{T}) \quad\Longleftrightarrow \quad \alpha\in \Omega^j(N), \beta\in \Omega^{j+1}(N).
$$
With $\partial_{\mathbb{T} }$ the canonical vector field on $\mathbb{T}$, we have the differential { $d_{\mathbb T}=d+\iota_{\partial_{\mathbb{T}}}$ defined by
$$
d_{\mathbb T}(\alpha+\vartheta_\mathbb{T} \wedge \beta) =d \alpha + \beta-   \vartheta_\mathbb{T} \wedge d\beta, \quad\text{ if $\alpha+ \vartheta_\mathbb{T}\wedge \beta$ is homogeneous,}
$$}
finally turning $\Omega_{\mathbb{T}}(N\times \mathbb{T})$ into a unital DGA.

\begin{Remark}\label{waldi} Given a manifold $X$ (possibly with boundary), the wedge product and the de Rham differential is continuous with respect to the canonical locally convex structure on $\Omega(X)$ \cite{waldmann}. In addition, if $B$ is a vector field on $X$ then the contraction
$$
\iota_B:\Omega(X)\longrightarrow \Omega(X)
$$ 
is continuous, and if $Y$ is another manifold (possibly with boundary) and if $\Psi:X\to Y$ is a smooth map, then the pullback map 
$$
\Psi^*: \Omega(Y)\longrightarrow \Omega(X)
$$ 
is continuous \cite{waldmann}. 
\end{Remark}

For every continuous seminorm $\varepsilon$ on $\Omega(N)$ we get a seminorm $\varepsilon^{\mathbb{T}}$ on $\Omega_{\mathbb{T}}(N\times \mathbb{T})$ by setting
$$
\varepsilon^{\mathbb{T}}(\alpha+\vartheta_\mathbb{T} \wedge \beta):=\varepsilon(\alpha)+\varepsilon(\beta)
$$
 In view of the formula $d_{\mathbb{T}}$, the space $\Omega_{\mathbb{T}}(N\times \mathbb{T})$ then becomes a locally convex unital DGA (by remark \ref{waldi}) in terms of the $\varepsilon^{\mathbb{T}}$'s. As a consequence, we get the super complexes

\begin{align}\label{ass33}
 &\mathscr{C}^{+}(\Omega_{\mathbb{T}}(N\times \mathbb{T}))\xrightarrow{b+B}\mathscr{C}^{-}(\Omega_{\mathbb{T}}(N\times \mathbb{T}))\xrightarrow{b+B}\mathscr{C}^{+}(\Omega_{\mathbb{T}}(N\times \mathbb{T})),\\\label{ass2}
&\mathscr{N}^{+}(\Omega_{\mathbb{T}}(N\times \mathbb{T}))\xrightarrow{b+B} \mathscr{N}^{-}(\Omega_{\mathbb{T}}(N\times \mathbb{T}))\xrightarrow{b+B} \mathscr{N}^{+}(\Omega_{\mathbb{T}}(N\times \mathbb{T})),\\\label{pna0}
&\mathscr{C}^{+}_{\epsilon}(\Omega_{\mathbb{T}}(N\times \mathbb{T}))\xrightarrow{b+B}\mathscr{C}^{-}_{\epsilon}(\Omega_{\mathbb{T}}(N\times \mathbb{T}))\xrightarrow{b+B}\mathscr{C}^{+}(\Omega_{\mathbb{T}}(N\times \mathbb{T})),\\\label{pna}
& \mathscr{N}_{\epsilon}^{+}(\Omega_{\mathbb{T}}(N\times \mathbb{T}))\xrightarrow{b+B} \mathscr{N}^{-}_{\epsilon}(\Omega_{\mathbb{T}}(N\times \mathbb{T}))\xrightarrow{b+B} \mathscr{N}^{+}_{\epsilon}(\Omega_{\mathbb{T}}(N\times \mathbb{T})).
\end{align}

\section{Equivariant Chen integrals}\label{posmmaa2}

Let us consider a compact manifold $N$ without boundary, and the space $L N$ of smooth loops $\gamma:\mathbb{T}\to N$, where in the sequel we read $\mathbb{T}$ as $\mathbb{T}=[0,1]/\sim$. This becomes an infinite dimensional Fr\'{e}chet manifold which is locally modelled on the Fr\'{e}chet space $L\IR^{\dim N}$ of smooth loops $\mathbb{T}\to \IR^{\dim N}$. Then $LN$ carries a natural smooth $\mathbb{T}$-action, given by rotating each loop, and the fixed point set of this action is precisely $N\subset LN$, embedded as constant loops. Given $\gamma \in LN$ the tangent space $T_{\gamma} LN$ is given by linear space of smooth vector fields on $N$ along $\gamma$, that is,
$$
T_{\gamma} (LN) = \big\{X\in C^{\infty}(\mathbb{T},N): X(t)\in T_{\gamma(t)}N \text{ for all $t\in \mathbb{T}$}\big\},
$$
and the generator of the $\mathbb{T}$-action on $LN$ is the vector field $\gamma\mapsto \dot{\gamma}$ on $LN$. Let $\iota$ denote the contraction with respect to the latter vector field. In the sequel, we understand 
$$
\Omega(LN):=\bigoplus^{\infty}_{k=0}\Omega^k(LM).
$$
For fixed $s\in \mathbb{T}$ one has the diffeomorphism 
$$
\phi_s:LN\longrightarrow LN, \quad\gamma \longmapsto\gamma(s+\cdot) 
$$
induced by the $\mathbb{T}$-action, and one gets an induced operator 
$$
P:\Omega(LN)\longrightarrow  \Omega(LN), \quad\text{defined on $\Omega^k(LN)$ by $P\alpha:=\int^1_0\phi^{*}_s\iota \alpha\ d s$.}
$$
Then $P$ becomes a degree $-1$ derivation. In addition, there is the usual exterior derivative
$$
d:\Omega(LN)\longrightarrow  \Omega(LN),
$$
a degree $+1$ derivation. Taking only odd/even degree forms, one gets the superstructure $\Omega=\Omega^{+}(LN)\oplus \Omega^-(LN)$, and we get the super complex 
\begin{align}\label{penis}
  \Omega^{+}(LN)\xrightarrow{\>\>d+P\>\>}\Omega^{-}(LN)\xrightarrow{\>\>d+P\>\>} \Omega^{+}(LN) ,
\end{align}
called the \emph{equivariant de Rham complex of $LN$}. This complex does not carry much information, as the differential forms of interest, like the Bismut-Chern character below, are actually elements of 
$$
\prod^{\infty}_{k=0}\Omega^k(LN),\quad\text{rather than}\quad \Omega(LN)=\bigoplus^{\infty}_{k=0}\Omega^k(LN).
$$
Thus we are going to 'complete' $\Omega(LN)$ in some way. To this end, following Chen's approach \cite{chen} of constructing a smooth structure on $LN$ in terms of plots, we consider smooth maps $f:X\to LN $, where $X$ is a finite dimensional manifold (without boundary). Given a continuous seminorm $\varepsilon$ on $\Omega(X)$ we get an induced seminorm
{$$
\varepsilon_f(\omega):=\varepsilon(f^*\omega)\quad\text{on $\Omega(LN)$}.
$$}
The locally convex topology induced by the $\epsilon_f$'s is Hausdorff and we define $\widehat{\Omega}(LN)$ to be the completion of $\Omega(LN)$ with respect to this locally convex topology. The maps $d$, $P$ and the grading operator become continuous maps $\Omega(LN)\to \Omega(LN)$: indeed, the continuity of the grading map is trivial. The continuity of $d$ follows from 
$$
\varepsilon_f(d\omega)= \varepsilon(d [f^*\omega])\leq \varepsilon'(f^*\omega)=\varepsilon'_f(\omega)
$$
for some continuous seminorm $\varepsilon'$ on $\Omega(X)$, where we have used the continuity of $d:\Omega(X)\to \Omega(X)$. Finally, the continuity of $P$ follows easily from the continuity of $\iota$, which in turn follows from writing
$$
\varepsilon_f(\iota\omega)=\varepsilon(f^* [\iota \omega])= \varepsilon( r^* \iota_{\partial_{\mathbb{T}}}  \hat{f}^*j^* [ \omega])\leq \varepsilon'_{j\circ \hat{f}}(  \omega)
$$
for some continuous seminorm $\varepsilon'$ on $\Omega(X\times \mathbb{T})$, where 
$$
r:X\longrightarrow  X\times \mathbb{T},\quad  j:N\longrightarrow LN
$$
are the canonical embeddings, and 
$$
\hat{f}:X\times \mathbb{T}\longrightarrow N
$$ 
the map induced by $f:X\to LN$, and where we have used Remark \ref{waldi} (the continuity of $r^* \iota_{\partial_{\mathbb{T}}}$, which implies the existence of $\varepsilon'$).\\
We end up with the super complex 
\begin{align}\label{penis}
  \widehat{\Omega}^{+}(LN)\xrightarrow{\>\>d+P\>\>}\widehat{\Omega}^{-}(LN)\xrightarrow{\>\>d+P\>\>} \widehat{\Omega}^{+}(LN),
\end{align}
called the \emph{completed equivariant de Rham complex of $LN$}. The corresponding homology groups are denoted by $\widehat{\mathsf{H}}^\pm_\mathbb{T}(LN)$.\vspace{1mm}

Given $t\in \mathbb{T}$ and $\alpha\in \Omega^k(N)$ one denotes with $\alpha(t)\in \Omega^k(LN)$ the form obtained by pulling $\alpha$ back with respect to the evaluation map $\gamma \mapsto \gamma(t)$. With this notation at hand, one has the \emph{equivariant Chen integral} map  
$$
\rho:\mathscr{C}(\Omega_{\mathbb{T}}(N\times \mathbb{T}))\longrightarrow
\Omega(LN),
$$
which is defined by{
\begin{align*}
&\rho\left\langle (\alpha_0+ \vartheta_\mathbb{T}\wedge \beta_0) \otimes \dots \otimes(\alpha_{n}+\vartheta_\mathbb{T}\wedge \beta_{n})  \right\rangle\\
&:=
\int_0^1 ds \phi^*_s \int_{\Delta_n}\!\!\alpha_0(0)\wedge (\iota\alpha_1(t_1)-\beta_1(t_1) )\wedge\cdots\wedge  (\iota\alpha_n(t_{n}) -\beta_n(t_{n})  ) \ d t_1\cdots d  t_{n},
\end{align*}
where
$$
\Delta_n=\{0\leq t_1\leq \dots\leq t_{ n}\leq 1\}\subset \IR^n
$$
denotes the standard $n$-simplex. We will also write
\begin{align*}
&\rho\left\langle (\alpha_0+ \vartheta_\mathbb{T}\wedge \beta_0) \otimes \dots \otimes(\alpha_{n}+\vartheta_\mathbb{T}\wedge \beta_{n})  \right\rangle\\
&=\int_0^1 ds \phi^*_s  \tilde\rho\left\langle (\alpha_0+ \vartheta_\mathbb{T}\wedge \beta_0) \otimes \dots \otimes(\alpha_{n}+\vartheta_\mathbb{T}\wedge \beta_{n})  \right\rangle.
\end{align*}
}

We collect the essential properties of $\rho$ in the following proposition:

\begin{Proposition} The map $\rho$ is a continuous morphism of super complexes 
\begin{align}\label{awayy}
 \rho: \mathscr{C}(\Omega_{\mathbb{T}}(N\times \mathbb{T}))\longrightarrow
\Omega(LN),
\end{align}
which in turn descends to a continuous map of super complexes
\begin{align}\label{awayy2}
 \rho : \mathscr{N}(\Omega_{\mathbb{T}}(N\times \mathbb{T}))\longrightarrow
\Omega(LN).
\end{align}
In particular, by density, we obtain the continuous maps of super complexes
$$
\rho : \mathscr{C}_{\epsilon}(\Omega_{\mathbb{T}}(N\times \mathbb{T}))\longrightarrow
\widehat{\Omega}(LN),\quad \rho : \mathscr{N}_{\epsilon}(\Omega_{\mathbb{T}}(N\times \mathbb{T}))\longrightarrow
\widehat{\Omega}(LN).
$$
\end{Proposition}

\begin{proof} i) The fact that (\ref{awayy}) is a map of superspaces follows easily from observing that 
\begin{align*}
&\mathscr{C}^+(\Omega_{\mathbb{T}}(N\times \mathbb{T}))=\bigoplus^{\infty}_{j=0}\mathscr{C}^{2j}(\Omega_{\mathbb{T}}(N\times \mathbb{T})),\\
&\mathscr{C}^-(\Omega_{\mathbb{T}}(N\times \mathbb{T}))=\bigoplus^{\infty}_{j=0}\mathscr{C}^{2j+1}(\Omega_{\mathbb{T}}(N\times \mathbb{T})),
\end{align*}
where 
\begin{align*}
&\mathscr{C}^k(\Omega_{\mathbb{T}}(N\times \mathbb{T}))\\
&=\bigoplus^{\infty}_{r=0}\bigoplus_{l_0+\cdots+l_r=k+r}\Omega^{l_0}_{\mathbb{T}}(N\times \mathbb{T}))\otimes \underline{\Omega^{l_1}_{\mathbb{T}}(N\times \mathbb{T}))}\otimes\cdots\otimes\underline{\Omega^{l_r}_{\mathbb{T}}(N\times \mathbb{T}))}, 
\end{align*}
and that $\rho$ maps $\mathscr{C}^k(\Omega_{\mathbb{T}}(N\times \mathbb{T}))\to \Omega^k(LN)$. \\
ii) { Next we show that $\rho (b+B)= (d+P)\rho$. Setting $\omega_j=\alpha_j+\vartheta_{\mathbb T}\wedge \beta_j$, we first notice 
\begin{align}\label{rhob}
 \tilde\rho b \left\langle \omega_0 \otimes \cdots \otimes \omega_n\right\rangle=
 &\tilde \rho\left\langle  d_{\mathbb T}\omega_0\otimes \cdots \otimes \omega_{j-1} \otimes \omega_j \otimes \omega_{j+1}\otimes \cdots \otimes \omega_n\right\rangle \cr
 &-\tilde\rho\left\langle \sum_{j=1}^n (-1)^{r_{j-1}} \omega_0\otimes \cdots \otimes \omega_{j-1} \otimes d_{\mathbb T}\omega_j \otimes \omega_{j+1}\otimes \cdots \otimes \omega_n\right\rangle\cr
 &-\tilde\rho\left\langle \sum_{j=0}^{n-1} (-1)^{r_j} \omega_0\otimes \cdots \otimes \omega_{j-1} \otimes \omega_j\wedge \omega_{j+1} \otimes \omega_{j+2} \otimes \cdots \otimes \omega_n\right\rangle\cr
 & +(-1)^{(j_n-1)r_{n-1}} \tilde\rho\left\langle \omega_n \wedge \omega_0 \otimes \omega_1 \otimes \cdots \otimes \omega_{n-1}\right\rangle.
\end{align}
The first two lines give
{\small
\begin{align*}
&\int_{\Delta_n} (d\alpha_0(0)+\beta_0(0))\wedge (\iota\alpha_1(t_1)-\beta_1(t_1))\wedge \cdots \wedge (\iota\alpha_n(t_n)-\beta_n(t_n))d^nt \cr
 &-\sum_{j=1}^n (-1)^{r_{j-1}}\int_{\Delta_n} \alpha_0(0) (\iota\alpha_1(t_1)-\beta_1(t_1)) \wedge \cdots \wedge  (\iota\alpha_{j-1}(t_{j-1})-\beta_{j-1}(t_{j-1})) \wedge \cr
& \quad \wedge (\iota d\alpha_j(t_{j})+\iota\beta_{j}(t_{j-1})+d\beta_j(t_j))\wedge (\iota\alpha_{j+1}(t_{j+1})-\beta_{j+1}(t_{j+1}))\wedge \cdots \wedge(\iota\alpha_n(t_n)-\beta_n(t_{n})) d^nt,
\end{align*}
}
where $d^n t=dt_1\cdots dt_n.$
Using that 
\begin{align*}
 \Delta_n=\{(t_1,t_2,\ldots,t_n): 0\leq t_1\leq \ldots \leq t_{j-1} \leq t_j \leq t_{j+1} \leq \ldots \leq t_n \},
\end{align*}
and that
\begin{align*}
 \iota d\alpha_j (t_j)=\frac d{dt_j} \alpha_j (t_j) - d\iota \alpha_j (t_j),
\end{align*}
it can be rewritten as
{\small
\begin{align*}
&\int_{\Delta_n} (d\alpha_0(0)+\beta_0(0))\wedge (\iota\alpha_1(t_1)-\beta_1(t_1))\wedge \cdots \wedge (\iota\alpha_n(t_n)-\beta_n(t_n))d^nt \cr
 &+\sum_{j=1}^n (-1)^{r_{j-1}}\int_{\Delta_n} \alpha_0(0) (\iota\alpha_1(t_1)-\beta_1(t_1)) \wedge \cdots \wedge  (\iota\alpha_{j-1}(t_{j-1})-\beta_{j-1}(t_{j-1})) \wedge \cr
& \quad \wedge d(\iota \alpha_j(t_{j})-\beta_j(t_j))\wedge (\iota\alpha_{j+1}(t_{j+1})-\beta_{j+1}(t_{j+1}))\wedge \cdots \wedge(\iota\alpha_n(t_n)-\beta_n(t_{n})) d^nt \cr
 &-\sum_{j=1}^n (-1)^{r_{j-1}}\int_{\Delta_n} \alpha_0(0) (\iota\alpha_1(t_1)-\beta_1(t_1)) \wedge \cdots \wedge  (\iota\alpha_{j-1}(t_{j-1})-\beta_{j-1}(t_{j-1})) \wedge \cr
& \quad \wedge \frac{d}{dt_j} \alpha_j(t_j)\wedge (\iota\alpha_{j+1}(t_{j+1})-\beta_{j+1}(t_{j+1}))\wedge \cdots \wedge(\iota\alpha_n(t_n)-\beta_n(t_{n})) d^nt \cr
 &-\sum_{j=1}^n (-1)^{r_{j-1}}\int_{\Delta_n} \alpha_0(0) (\iota\alpha_1(t_1)-\beta_1(t_1)) \wedge \cdots \wedge  (\iota\alpha_{j-1}(t_{j-1})-\beta_{j-1}(t_{j-1})) \wedge \cr
& \quad \wedge \iota \beta_j(t_j)\wedge (\iota\alpha_{j+1}(t_{j+1})-\beta_{j+1}(t_{j+1}))\wedge \cdots \wedge(\iota\alpha_n(t_n)-\beta_n(t_{n})) d^nt.
\end{align*}
}

The first two (three) lines give 
{\small
\begin{align}
d\tilde\rho \left\langle \omega_0 \otimes \cdots \otimes \omega_n\right\rangle +\int_{\Delta_n} \beta_0(0)\wedge (\iota\alpha_1(t_1)-\beta_1(t_1))\wedge \cdots \wedge (\iota\alpha_n(t_n)-\beta_n(t_n))d^nt ,
\end{align}
}

while the third (fourth and fifth) line can be integrated in $t_j$ from $t_{j-1}$ to $t_{j+1}$ thus getting
{\small
\begin{align}
& d\tilde\rho \left\langle \omega_0 \otimes \cdots \otimes \omega_n\right\rangle+\int_{\Delta_n} \beta_0(0)\wedge (\iota\alpha_1(t_1)-\beta_1(t_1))\wedge \cdots \wedge (\iota\alpha_n(t_n)-\beta_n(t_n))d^nt\cr
&-\sum_{j=1}^{n-1} (-1)^{r_{j-1}}\int_{\Delta_{n-1}} \alpha_0(0) \wedge(\iota\alpha_1(t_1)-\beta_1(t_1)) \wedge \cdots \wedge  (\iota\alpha_{j-1}(t_{j-1})-\beta_{j-1}(t_{j-1})) \wedge \cr
& \quad \wedge \alpha_j(t_{j+1})\wedge (\iota\alpha_{j+1}(t_{j+1})-\beta_{j+1}(t_{j+1}))\wedge \cdots \wedge(\iota\alpha_n(t_n)-\beta_n(t_{n})) d^nt_j \cr
&- (-1)^{r_{n-1}}\int_{\Delta_{n-1}} \alpha_0(0) \wedge (\iota\alpha_1(t_1)-\beta_1(t_1)) \wedge \cdots  \wedge (\iota\alpha_{n-1}(t_{n-1})-\beta_{n-1}(t_{n-1})) \wedge \alpha_n(1) d^n t_n\cr
&+\sum_{j=2}^{n} (-1)^{r_{j-1}}\int_{\Delta_{n-1}} \alpha_0(0) \wedge(\iota\alpha_1(t_1)-\beta_1(t_1)) \wedge \cdots \wedge  (\iota\alpha_{j-1}(t_{j-1})-\beta_{j-1}(t_{j-1})) \wedge \cr
& \quad \wedge \alpha_j(t_{j-1})\wedge (\iota\alpha_{j+1}(t_{j+1})-\beta_{j+1}(t_{j+1}))\wedge \cdots \wedge(\iota\alpha_n(t_n)-\beta_n(t_{n})) d^nt_j \cr
&+ (-1)^{r_{0}}\int_{\Delta_{n-1}} \alpha_0(0)\wedge \alpha_1(0) \wedge (\iota\alpha_{2}(t_{2})-\beta_{2}(t_{2})) \wedge \cdots \wedge(\iota\alpha_n(t_n)-\beta_n(t_{n})) d^nt_1 \cr
&-\sum_{j=1}^n (-1)^{r_{j-1}}\int_{\Delta_n} \alpha_0(0)\wedge (\iota\alpha_1(t_1)-\beta_1(t_1)) \wedge \cdots \wedge  (\iota\alpha_{j-1}(t_{j-1})-\beta_{j-1}(t_{j-1})) \wedge \cr
& \quad \wedge \iota \beta_j(t_j)\wedge (\iota\alpha_{j+1}(t_{j+1})-\beta_{j+1}(t_{j+1}))\wedge \cdots \wedge(\iota\alpha_n(t_n)-\beta_n(t_{n})) d^nt,
\end{align}\label{somma}
}
where $d^n t_j =dt_1\cdots dt_{j-1} dt_{j+1}\cdots dt_n$. If in the fourth sum of integrals we change the summation variable from $j$ to $j+1$, then make the change of variable $t_j\rightarrow t_{j+1}$, and put it together with the second sum of integrals, 
after noting that $(-1)^{r_{j-1}}(-1)^{j_j}=-(-1)^{r_j}$, then summing the fourth and the second integrals we get
{\small
\begin{align*}
&-\sum_{j=1}^{n-1} (-1)^{r_{j-1}}\int_{\Delta_{n-1}} \alpha_0(0) \wedge(\iota\alpha_1(t_1)-\beta_1(t_1)) \wedge \cdots \wedge  (\iota\alpha_{j-1}(t_{j-1})-\beta_{j-1}(t_{j-1})) \wedge \cr
& \quad \wedge \left[\alpha_j(t_{j+1})\wedge (\iota\alpha_{j+1}(t_{j+1})-\beta_{j+1}(t_{j+1}))\right]\wedge \cdots \wedge(\iota\alpha_n(t_n)-\beta_n(t_{n})) d^nt_j \cr
&+\sum_{j=1}^{n-1} (-1)^{r_{j}}\int_{\Delta_{n-1}} \alpha_0(0) \wedge(\iota\alpha_1(t_1)-\beta_1(t_1)) \wedge \cdots \wedge  (\iota\alpha_{j-1}(t_{j-1})-\beta_{j-1}(t_{j-1})) \wedge \cr
& \quad \wedge  \left[(\iota\alpha_{j}(t_{j+1})-\beta_{j}(t_{j+1}))\wedge \alpha_{j+1}(t_{j+1})\right]\wedge \cdots \wedge(\iota\alpha_n(t_n)-\beta_n(t_{n})) d^nt_j \cr
&=\sum_{j=1}^{n-1} (-1)^{r_{j}}\int_{\Delta_{n-1}} \alpha_0(0) \wedge(\iota\alpha_1(t_1)-\beta_1(t_1)) \wedge \cdots \wedge  (\iota\alpha_{j-1}(t_{j-1})-\beta_{j-1}(t_{j-1})) \wedge \cr
& \quad \wedge  \left[(\iota\alpha_{j}(t_{j+1})-\beta_{j}(t_{j+1}))\wedge \alpha_{j+1}(t_{j+1})+(-1)^{j_j-1} \alpha_j(t_{j+1})\wedge (\iota\alpha_{j+1}(t_{j+1})-\beta_{j+1}(t_{j+1}))
\right]\wedge \cr
&\quad \wedge \cdots \wedge(\iota\alpha_n(t_n)-\beta_n(t_{n})) d^nt_j \cr
&=\sum_{j=1}^{n-1} (-1)^{r_{j}} \tilde\rho \langle\omega_0 \otimes \cdots \otimes  \omega_{j-1} 
\otimes \omega_{j}\wedge \omega_{j+1} \otimes  \omega_{j+2} \otimes \cdots \otimes \omega_n \rangle,
\end{align*}
}
which including the fifth integral in (\ref{somma}) becomes
\begin{align*}
\tilde\rho \left\langle\sum_{j=0}^{n-1} (-1)^{r_{j}}  \omega_0 \otimes \cdots \otimes  \omega_{j-1} 
\otimes \omega_{j}\wedge \omega_{j+1} \otimes  \omega_{j+2} \otimes \cdots \otimes \omega_n \right\rangle.
\end{align*}
This cancels the second line of (\ref{rhob}). After noting that $\alpha_n(1)=\alpha_n(0)$, we see that the third integral in (\ref{somma}) is just
\begin{align*}
 -(-1)^{(j_n-1)r_{n-1}} \tilde\rho \left\langle \omega_n \wedge \omega_0 \otimes \omega_1 \otimes \cdots \otimes \omega_{n-1}\right\rangle,
\end{align*}
which cancels the third line of (\ref{rhob}). Thus, we get
\begin{align}\label{rhobidrho}
&\tilde\rho b\left\langle \omega_0 \otimes \cdots \otimes \omega_n \right\rangle=d\tilde\rho  \left\langle \omega_0 \otimes \cdots \otimes \omega_n\right\rangle\cr
&\quad+\int_{\Delta_n} \beta_0(0)\wedge (\iota\alpha_1(t_1)-\beta_1(t_1))\wedge \cdots \wedge (\iota\alpha_n(t_n)-\beta_n(t_n))d^nt\cr
&\quad-\sum_{j=1}^n (-1)^{r_{j-1}}\int_{\Delta_n} \alpha_0(0)\wedge (\iota\alpha_1(t_1)-\beta_1(t_1)) \wedge \cdots \wedge  (\iota\alpha_{j-1}(t_{j-1})-\beta_{j-1}(t_{j-1})) \wedge \cr
& \qquad \wedge \iota \beta_j(t_j)\wedge (\iota\alpha_{j+1}(t_{j+1})-\beta_{j+1}(t_{j+1}))\wedge \cdots \wedge(\iota\alpha_n(t_n)-\beta_n(t_{n})) d^nt.
\end{align}
Now, let us consider 
{\small
\begin{align}
 P\tilde\rho  \left\langle \omega_0 \otimes \cdots \otimes \omega_n\right\rangle=& \int_I ds \phi_s^* \iota \int_{\Delta_n} \alpha_0(0) \wedge (\iota\alpha_1(t_1)-\beta_1(t_1)) \wedge \cdots \wedge (\iota\alpha_n(t_n)-\beta_n(t_n)) d^n t \cr
 =\int_{I\times \Delta_n} \iota\alpha_0(s)\wedge &(\iota\alpha_1(t_1+s)-\beta_1(t_1+s)) \wedge \cdots \wedge (\iota\alpha_n(t_n+s)-\beta_n(t_n+s)) d^n t ds\cr
 -\sum_{j=1}^n (-1)^{r_{j-1}}\int_I ds \phi_s^* &\int_{\Delta_n}  \alpha_0(0)\wedge (\iota\alpha_1(t_1)-\beta_1(t_1)) \wedge \cdots \wedge  (\iota\alpha_{j-1}(t_{j-1})-\beta_{j-1}(t_{j-1})) \wedge \cr
& \wedge \iota \beta_j(t_j)\wedge (\iota\alpha_{j+1}(t_{j+1})-\beta_{j+1}(t_{j+1}))\wedge \cdots \wedge(\iota\alpha_n(t_n)-\beta_n(t_{n})) d^nt,
\end{align}\label{pirho}
}
where now $I$ must be identified with the circle $\mathbb T$, and where we used that 
$$
\iota(\iota\alpha_{k}(t_{k})-\beta_{k}(t_{k}))=-\iota \beta_{k}(t_{k}).
$$
Now, for any given choice of $\bar t=(t_1,\ldots,t_n)$ such that $0\leq t_1 \leq \cdots \leq t_n \leq 1$, we can understand $\mathbb T$ as the union of almost everywhere $n+1$ disjoint intervals
defined by
\begin{align*}
 I_j(\bar t)=\{ s\in \mathbb T| t_{j-1} +s\leq 1,\   t_{j} +s-1 \geq 0 \}, \quad j=1,\ldots, n+1.
\end{align*}
We see that
\begin{align*}
 D_j=\{ I_j(\bar t) \times \bar t\ |\ \bar t \in \Delta_n \}
\end{align*}
is a $(n+1)$-simplex for any given $j$, and
\begin{align*}
\bigcup_{j=1}^{n+1} D_j=I\times \Delta_n
\end{align*}
while $D_j\cap D_k$ has zero measure if $j\neq k$. Therefore, 
{\small
\begin{align*}
&\int_{I\times \Delta_n} \iota\alpha_0(s)\wedge (\iota\alpha_1(t_1+s)-\beta_1(t_1+s)) \wedge \cdots \wedge (\iota\alpha_n(t_n+s)-\beta_n(t_n+s)) d^n t ds\cr
&=\int_{I\times \Delta_n} \beta_0(s)\wedge (\iota\alpha_1(t_1+s)-\beta_1(t_1+s)) \wedge \cdots \wedge (\iota\alpha_n(t_n+s)-\beta_n(t_n+s)) d^n t ds\cr
&+\int_{I\times \Delta_n} (\iota\alpha_0(s)-\beta_0)\wedge (\iota\alpha_1(t_1+s)-\beta_1(t_1+s)) \wedge \cdots \wedge (\iota\alpha_n(t_n+s)-\beta_n(t_n+s)) d^n t ds\cr
&=\int_I ds \phi^*_s \int_{\Delta_n} \beta_0(0)\wedge (\iota\alpha_1(t_1)-\beta_1(t_1)) \wedge \cdots \wedge (\iota\alpha_n(t_n)-\beta_n(t_n)) d^n t ds\cr
&+\sum_{j=1}^{n+1} \int_{D_j} (\iota\alpha_0(s)-\beta_0(s))\wedge (\iota\alpha_1(t_1+s)-\beta_1(t_1+s)) \wedge \cdots \wedge (\iota\alpha_n(t_n+s)-\beta_n(t_n+s)) d^n t ds.
\end{align*}
}
Now, for any given $j$ we introduce the variables
\begin{align*}
\tau_k&=t_{j+k-1}+s-1, \quad k=1,\ldots, n+1-j, \cr
\tau_{n+2-j}&=s, \cr
\tau_k&=t_{k+j-n-2}+s,\quad k= n+3-j, \ldots, n+1 \quad (\mbox{if }j\geq 2). 
\end{align*}
In this coordinates we have
\begin{align*}
 D_j=\{(\tau_1,\ldots,\tau_{n+1})| 0\leq \tau_1\leq \cdots \leq \tau_{n+1}\leq 1\}\equiv \Delta_{n+1}, \quad d^n t ds=d^{n+1}\tau,
\end{align*}
and
{\small 
\begin{align*}
 &(\iota\alpha_0(s)-\beta_0(s))\wedge (\iota\alpha_1(t_1+s)-\beta_1(t_1+s)) \wedge \cdots \wedge (\iota\alpha_n(t_n+s)-\beta_n(t_n+s)) \cr
 &=(-1)^{r_{j-1}(r_n-r_j)}1\wedge (\iota\alpha_j(\tau_1)-\beta_j(\tau_1))\wedge \cdots \wedge (\iota\alpha_n (\tau_{n-j+1})-\beta_n (\tau_{n-j+1})) \wedge \cr
 &\quad \wedge (\iota\alpha_0(\tau_{n-j+2})-\beta_0(\tau_{n-j+2}))\wedge \cdots \wedge (\iota\alpha_{j-1}(\tau_{n+1})-\beta_{j-1}(\tau_{n+1})).
\end{align*}
}
Integrating over $D_j=\Delta_{n+1}$ it becomes
{\small 
\begin{align*}
 &\int_{D_j}(\iota\alpha_0(s)-\beta_0(s))\wedge (\iota\alpha_1(t_1+s)-\beta_1(t_1+s)) \wedge \cdots \wedge (\iota\alpha_n(t_n+s)-\beta_n(t_n+s))\cr
 &=\rho\left\langle (-1)^{r_{j-1}(r_n-r_j)}1\otimes \omega_j \otimes \cdots \otimes \omega_n \otimes \omega_0\otimes \cdots \otimes \omega_{j-1}\right\rangle,
\end{align*}
}
and after summation over $j$ we finally get 
{\small
\begin{align*}
&P\tilde\rho\left\langle \omega_0 \otimes \cdots \otimes \omega_n\right\rangle=\tilde\rho B\left\langle \omega_0 \otimes \cdots \otimes \omega_n\right\rangle \cr
&+\int_I ds \phi^*_s \int_{\Delta_n} \beta_0(0)\wedge (\iota\alpha_1(t_1)-\beta_1(t_1)) \wedge \cdots \wedge (\iota\alpha_n(t_n)-\beta_n(t_n)) d^n t ds \cr
& -\sum_{j=1}^n (-1)^{r_{j-1}}\int_I ds \phi_s^* \int_{\Delta_n}  \alpha_0(0)\wedge (\iota\alpha_1(t_1)-\beta_1(t_1)) \wedge \cdots \wedge  (\iota\alpha_{j-1}(t_{j-1})-\beta_{j-1}(t_{j-1})) \wedge \cr
& \wedge \iota \beta_j(t_j)\wedge (\iota\alpha_{j+1}(t_{j+1})-\beta_{j+1}(t_{j+1}))\wedge \cdots \wedge(\iota\alpha_n(t_n)-\beta_n(t_{n})) d^nt.
\end{align*}
}
Notice that the second and third lines here are the means over $\mathbb T$ of the corresponding terms in (\ref{rhobidrho}). After taking the mean of both expressions and subtracting each other, we finally
get $\rho (b+B)= (d+P)\rho$ as desired.\\
iii) We now prove that $\tilde \rho$ vanishes on $\mathscr{D}(\Omega_{\mathbb{T}}(N\times \mathbb{T}))$. This implies that $\rho$ vanishes on $\mathscr{D}(\Omega_{\mathbb{T}}(N\times \mathbb{T}))$, too. For elements of the form (\ref{titty0}) the assertion immediately follows from the fact that $\iota f(t)=0$, as $f(t)$ is a zero form. So, let us consider an element of the form (\ref{titty}). Since (recall that $f$ is constant over $\mathbb T$)
$$
\iota d f(t)=\frac {d}{dt} f(t), 
$$
and $df=d_{\mathbb T} f$, we can write
{\small
\begin{align*}
\tilde\rho(&\left\langle \omega_{0}\otimes \cdots\otimes \omega_{r-1} f\otimes\omega_{r+1}\otimes \cdots\otimes \omega_{n}\right\rangle+\left\langle \omega_{0}\otimes \cdots\otimes \omega_{r-1}\otimes df\otimes\omega_{r+1}\otimes \cdots\otimes \omega_{n}\right\rangle\cr
&\quad-\left\langle \omega_{0}\otimes \cdots\otimes \omega_{r-1}\otimes f\omega_{r+1}\otimes \cdots\otimes \omega_{n}\right\rangle)\cr
&=\int_{\Delta_{n-1}} \alpha_0(0) \wedge \cdots \wedge (\iota\alpha_{r-1}(t_{r-1})f(t_{r-1})-\beta_{r-1}(t_{r-1})f(t_{r-1}))
\wedge (\iota\alpha_{r+1}(t_{r+1})-\beta_{r+1}(t_{r+1})) \wedge \cr
&\qquad\ \wedge \cdots \wedge (\iota\alpha_{n}(t_{n})-\beta_{n}(t_{n})) d^n t_r \cr
&-\int_{\Delta_{n-1}} \alpha_0(0) \wedge \cdots \wedge (\iota\alpha_{r-1}(t_{r-1})-\beta_{r-1}(t_{r-1}))
\wedge (f(t_{r+1})\iota\alpha_{r+1}(t_{r+1})-f(t_{r+1})\beta_{r+1}(t_{r+1})) \wedge \cr
&\qquad\ \wedge \cdots \wedge (\iota\alpha_{n}(t_{n})-\beta_{n}(t_{n})) d^n t_r\cr
&+\int_{\Delta_n} \alpha_0(0) \wedge \cdots \wedge (\iota\alpha_{r-1}(t_{r-1})-\beta_{r-1}(t_{r-1})) \wedge \frac {d}{dt_r} f(t_r)
\wedge (\iota\alpha_{r}(t_{r})-\beta_{r}(t_{r})) \wedge \cr
&\qquad\ \wedge \cdots \wedge (\iota\alpha_{n}(t_{n})-\beta_{n}(t_{n})) d^nt.
\end{align*}
After integrating $t_r$ from $t_{r-1}$ to $t_{r+1}$ in the last term, we get exactly zero.}\\
v)} It remains to check the continuity of (\ref{awayy}), which easily follow from the continuity of $\tilde\rho$. To see the latter, let $X$ be a smooth manifold (without boundary), let $\varepsilon$ be a continuous seminorm on $\Omega(X)$, and let $f:X\to LN$ 
be smooth. For $s\in\mathbb{T}$ let $r_{s}$ denote the embedding
$$
X \longrightarrow  X\times \mathbb T, \ x\longmapsto (x,s).
$$
Then we have 
\begin{align*}
&\varepsilon_f\left(\tilde\rho\left\langle (\alpha_0+\vartheta_\mathbb{T}\wedge\beta_0) \otimes \dots \otimes(\alpha_{n}+\vartheta_\mathbb{T}\wedge\beta_{n})  \right\rangle \right)\\
&\leq \int_{\Delta_n} \varepsilon( f^*[\alpha_0(0)] )\prod^n_{i=1} \varepsilon\big(f^*[\iota\alpha_i(t_i)-\beta_i(t_i) ] \big)\ d t_1\cdots d  t_{n}\\
&=\int_{\Delta_n} \varepsilon( r_0^*\hat{f}^*\alpha_0 )\prod^n_{i=1} \varepsilon\big(r_{t_i}^*\iota_{\partial_{\mathbb{T}}}\hat{f}^*\alpha_i-r_{t_i}^*\hat{f}^*\beta_i  \big)\ d t_1\cdots d  t_{n}\\
&\leq\int_{\Delta_n} \varepsilon( r_0^*\hat{f}^*\alpha_0 )\prod^n_{i=1}\Big( \varepsilon\big(r_{t_i}^*\iota_{\partial_{\mathbb{T}}}\hat{f}^*\alpha_i\big)+\varepsilon\big(r_{t_i}^*\hat{f}^*\beta_i  \big)\Big)\ d t_1\cdots d  t_{n}\\
&\leq\int_{\Delta_n} \tilde{\varepsilon}( \alpha_0 )\prod^n_{i=1}\Big( \tilde{\varepsilon}(\alpha_i)+\tilde{\varepsilon}(\beta_i  )\Big)\ d t_1\cdots d  t_{n}
\\
&\leq \frac{1}{n!} \prod^n_{i=0}\Big( \tilde{\varepsilon}(\alpha_i)+\tilde{\varepsilon}(\beta_i  )\Big)= \frac{1}{n!} \tilde{\varepsilon}^{\mathbb{T}}_n\Big ((\alpha_0+\vartheta_\mathbb{T}\wedge\beta_0) \otimes \dots 
\otimes(\alpha_{n}+\vartheta_\mathbb{T}\wedge\beta_{n}) \Big),
\end{align*}
for some continuous seminorm $\tilde{\varepsilon}$ on $\Omega(N)$. This estimate shows the continuity of $\tilde\rho$ and completes the proof.

\end{proof}


\section{Construction of cycles in $\mathscr{N}^{-}_{\epsilon}(\Omega_{\mathbb{T}}(M\times \mathbb{T}))$ and the induced cycles in $\widehat{\Omega}^-(LM)$}

Let now $M$ be a compact manifold (possibly with boundary). Given $g\in C^{\infty}(M, U(l\times l;\IC))$ our aim is to construct a canonically given element 
$$
\mathrm{Ch}^-(g)\in \mathscr{C}^{-}_{\epsilon}(\Omega_{\mathbb{T}}(M\times \mathbb{T}))
$$ 
with $(b+B)\mathrm{Ch}^-(g)=0$ in the Chen normalized complex. To this end, let $I:=[0,1]$ and denote the canonical vector field on $I$ with $\partial_I$. We denote the canonical Maurer-Cartan form on $U(l\times l;\IC)$ by 
$$
\omega \in \Omega^1\big(U(l\times l;\IC), \mathrm{Mat}(l\times l;\IC)  \big).
$$
Then for all $s\in I$ we can form the covariant derivative $d+s\omega$ on the trivial vector bundle $U(l\times l;\IC)\times  \IC^l\to U(l\times l;\IC)$. Let 
$$
A^s \in \Omega^1\big(U(l\times l;\IC) , \mathrm{Mat}(l\times l;\IC)  \big)  ,\quad   R^s \in \Omega^2\big(U(l\times l;\IC), \mathrm{Mat}(l\times l;\IC)  \big)
 $$	
denote the connection $1$-form of $d+s\omega $ and the curvature of $d+s\omega $, respectively, and 
$$
\mathcal{A}^s :=A^s -\vartheta_\mathbb{T}\wedge R^s\in \Omega_{\mathbb{T}}\big(U(l\times l;\IC) \times \mathbb{T}, \mathrm{Mat}(l\times l;\IC)\big).
$$

We set
$$
A^s(g):= g^*A^s,\quad R^s_g:= g^*R^s,\quad \omega_g:= g^* \omega,
$$
so that $A^s(g)=s \omega_g$ and by the Maurer-Cartan equation $R^s_g=(s/2)\omega^2_g$. Then we can define 
$$
\mathcal{A}^s(g):=A^s_g -\vartheta_\mathbb{T}\wedge  R^s_g  \in \Omega_\mathbb{T}(M\times  \mathbb{T}, \mathrm{Mat}(l\times l;\IC)).
$$
By varying $s$, the forms $\mathcal{A}^s(g)$ induce a form
$$
\mathcal{A}(g) \in \Omega_\mathbb{T}(M\times  I \times \mathbb{T}, \mathrm{Mat}(l\times l;\IC))
$$
and we set
$$
\mathcal{B}(g):= \iota_{\partial_I}\mathcal{A}(g) \in \Omega_\mathbb{T}(M\times   I \times \mathbb{T}, \mathrm{Mat}(l\times l;\IC)).
$$
Then we can define
$$
 \mathcal{B}^s(g)\in \Omega_\mathbb{T}(M\times  \mathbb{T}, \mathrm{Mat}(l\times l;\IC)),
$$
to be the pullback of $\mathcal{B}(g)$ with respect to the embedding
$$
M\times  \mathbb{T}\longrightarrow M\times   I \times \mathbb{T}, \quad (x,t)\longmapsto (x,s,t).
$$

In fact, by a simple calculation one finds
\begin{align}\label{idid}
 \mathcal A^s(g)=s \omega_g +s(1-s)\vartheta_\mathbb{T}\wedge \omega^2_g, \quad \mathcal B^s(g)=-\vartheta_\mathbb{T}\wedge \omega_g,
 \end{align}
so that $\mathcal{B}^s(g)$ actually does not depend on $s$. With these preparations, we can define an element
$$
\mathrm{Ch}^-(g)=(\mathrm{Ch}^-_0(g) ,\mathrm{Ch}^-_1(g) ,\dots)\in \mathscr{C}(\Omega_{\mathbb{T}}(M\times \mathbb{T}))
$$
by setting 

$$
\mathrm{Ch}^-_n(g) := \mathrm{Tr}_n\left[ \int^1_0 1  \otimes \sum^n_{k=1}  {\mathcal A}^{s}(g)^{\otimes (k-1)} \otimes   {\mathcal B}^s(g)  \otimes  {\mathcal A}^{s}(g)^{\otimes (n-k)}ds\right],
$$
where given linear spaces $V_0,\dots,V_n$, and $ v^{(j)}\in \mathrm{Mat}(l\times l;V_i)$, $j=0,\dots, n$, the generalized trace is defined by  
$$
\mathrm{Tr}_{n}[v^{(0)}\otimes\cdots\otimes v^{(n)}]:=\sum_{i_0,\dots,i_n=1,\dots l} v^{(0)}_{i_0,i_1}\otimes v^{(1)}_{i_1,i_2} \otimes\cdots \otimes v^{(n)}_{i_{n},i_0}.
$$

We refer the reader to the paper \cite{simonssullivan} by Simons and Sullivan, where a construction of the usual odd Chern character $\mathrm{ch}^-(g)\in\Omega^-(M)$ (cf. formula (\ref{daaqd}) below) has been given that influenced our definition of $\mathrm{Ch}^-(g)$.

\begin{Theorem}\label{main} Let $M$ be a compact manifold, possibly with boundary.\\
a) One has 
$$
\mathrm{Ch}^-(g)\in \mathscr{C}^{-}_{\epsilon}(\Omega_{\mathbb{T}}(M\times \mathbb{T})),\quad\text{and $(b+B) \mathrm{Ch}^-(g)=0$ in $\mathscr{N}_{\epsilon}(\Omega_{\mathbb{T}}(M\times \mathbb{T}))$},
$$
in particular, $\mathrm{Ch}^-(g)$ induces a homology class 
$$
\big[\mathrm{Ch}^-(g)\big]\in \mathsf{HN}^{-}_{\epsilon}(\Omega_{\mathbb{T}}(M\times \mathbb{T})).
$$
b) The map
$$
\mathsf{K}^{-1}(M)\longrightarrow \mathsf{HN}^{-}_{\epsilon}(\Omega_{\mathbb{T}}(M\times \mathbb{T})),\quad [g]\longmapsto \big[\mathrm{Ch}^-(g)\big]
$$ 
is a well-defined group homomorphism.
\end{Theorem}

\begin{proof}a) It is easily seen that $\Gamma\mathrm{Ch}^-(g)=-\mathrm{Ch}^-(g)$. To show that 
$$
\mathrm{Ch}^-(g)\in \mathscr{C}^{-}_{\epsilon}(\Omega_{\mathbb{T}}(M\times \mathbb{T})),
$$
given a continuous seminorm $\varepsilon$ on $\Omega_{\mathbb{T}}(M\times \mathbb{T})$ set
$$
C_{\varepsilon}:=\sup_{s\in [0,1]}\max\Big(\varepsilon(1),\max_{i,j=1,\dots,l} \varepsilon(\mathcal{A}^{s}(g)_{ij}), \max_{i,j=1,\dots,l} \varepsilon(\mathcal{B}^{s}(g)_{ij})\Big).
$$
It is then easily checked that
\begin{align*}
\kappa_{\varepsilon}(\mathrm{Ch}^-(g))\leq  \sum_{n=0}^{\infty} n \frac{ (l^{2}C_{\varepsilon})^n}{\sqrt{n!}}<\infty. 
\end{align*}

It remains to prove 
$$
(b+B) \mathrm{Ch}^-(g)\in\mathscr{D}_{\epsilon}(\Omega_{\mathbb{T}}(M\times \mathbb{T})).
$$
In fact, 
$$
B\mathrm{Ch}^-(g)\in \mathscr{D}_{\epsilon}(\Omega_{\mathbb{T}}(M\times \mathbb{T})),
$$
as every $\left\langle \mathrm{Ch}^-_n(g) \right\rangle$ contains the $0$-form $1$ and so is of the form (\ref{titty0}) with $f=1$. It remains to show that 
$$
b\mathrm{Ch}^-(g)\in \mathscr{D}_{\epsilon}(\Omega_{\mathbb{T}}(M\times \mathbb{T})).
$$
 In order to see the latter, let us first notice that 
 \begin{align*}
 \left(b\mathrm{Ch}^-(g)\right)_n=\left(b\left\langle \mathrm{Ch}^-_n(g)\right\rangle\right)_n+\left(b\left\langle \mathrm{Ch}^-_{n+1}(g)\right\rangle \right)_n.
\end{align*}
Using (\ref{idid}) and the explicit definition of $b$, we get
{
\begin{align*}
&\left(b\left\langle \mathrm{Ch}^-_n(g)\right\rangle\right)_n \\
&=
-{\rm Tr}_n \left[\int_0^1  1 \otimes \sum_{k=1}^n  \sum_{l=0}^{k-2} \mathcal A^s(g)^{\otimes l} \otimes (-s^2 \omega_g^2) \otimes \mathcal A^s(g)^{\otimes (k-l-2)} \right. \cr
& \left. \phantom{mangialatuttalapastasciutta} 
\otimes (-\vartheta_\mathbb{T}\wedge \omega_g )\otimes  {\mathcal A}^s(g)^{\otimes (n-k)}  \ ds\right]\cr
&\quad+{\rm Tr}_n \left[\int_0^1  1 \otimes \sum_{k=1}^n  \sum_{l=0}^{n-k-1} \mathcal A^s(g)^{\otimes (k-1)} \otimes (-\vartheta_\mathbb{T}\wedge \omega_g ) \right. \cr
& \left. \phantom{mangialatuttalapastasciutta} 
\otimes \mathcal A^s(g)^{\otimes l} \otimes (-s^2\omega_g^2 ) \otimes  {\mathcal A}^s(g)^{\otimes (n-k-l-1)}  \ ds\right]\cr 
&\quad-{\rm Tr}_n \left[\int_0^1  1 \otimes \sum_{k=1}^n   \mathcal A^s(g)^{\otimes (k-1)} \otimes (\vartheta_\mathbb{T}\wedge \omega_g^2  +\omega_g) \otimes \mathcal A^s(g)^{\otimes (n-k)}  \ ds\right],
\end{align*}
}
and
{
\begin{align*}
&\left(b\left\langle \mathrm{Ch}^-_{n+1}(g)\right\rangle\right)_n \\
&=-{\rm Tr}_n \left[\int_0^1  1 \otimes \sum_{k=1}^n  \sum_{l=0}^{k-2} \mathcal A^s(g)^{\otimes l} \otimes (+s^2 \omega_g^2) \otimes \mathcal A^s(g)^{\otimes (k-l-2)} \right. \cr
& \left. \phantom{mangialatuttalapastasciutta} 
\otimes (-\vartheta_\mathbb{T}\wedge \omega_g )
\otimes  {\mathcal A}^s(g)^{\otimes (n-k)}  \ ds\right]\cr
&\quad+{\rm Tr}_n \left[\int_0^1  1 \otimes \sum_{k=1}^n  \sum_{l=0}^{n-k-1} \mathcal A^s(g)^{\otimes (k-1)} \otimes (-\vartheta_\mathbb{T}\wedge \omega_g ) \otimes \mathcal A^s(g)^{\otimes l} \right. \cr
& \left. \phantom{mangialatuttalapastasciutta} 
 \otimes (+s^2\omega_g^2 )
\otimes  {\mathcal A}^s(g)^{\otimes (n-k-l-1)}  \ ds\right]\cr 
&\quad-{\rm Tr}_n \left[\int_0^1  1 \otimes \sum_{k=1}^n   \mathcal A^s(g)^{\otimes (k-1)} \otimes (-2 s \vartheta_\mathbb{T}\wedge \omega_g^2 ) \otimes \mathcal A^s(g)^{\otimes (n-k)}  \ ds\right],
\end{align*}
}
whose sum is
\begin{align*}
& {\rm Tr}_n \left[\int_0^1  1 \otimes \sum_{k=1}^n  \mathcal A^s(g)^{\otimes (k-1)} \otimes \left(\frac {d }{ds} \mathcal A^s(g) \right) \otimes  {\mathcal A}^s(g)^{\otimes (n-k)}  \ ds\right]\\
& =
{\rm Tr}_n \left[\int_0^1  \frac {d }{ds} \left( 1 \otimes \mathcal A^s(g)^{\otimes n}\right) \ ds\right] = {\rm Tr}_n \left[ 1 \otimes \mathcal A^1(g)^{\otimes n} \right] -{\rm Tr}_n \left[ 1 \otimes \mathcal A^0(g)^{\otimes n} \right] .
\end{align*}
Thus, we finally have 
\begin{align*}
 (b\mathrm{Ch}^-(g))_n={\rm Tr}_n \left[ 1 \otimes \mathcal \omega_g^{\otimes n} \right], \qquad n=1,2,\ldots.
 \end{align*}
We now prove that 
$$
\left(\dots,{\rm Tr}_n \left[ 1 \otimes \mathcal \omega_g^{\otimes n} \right],\dots \right) \in \mathscr{D}_{\epsilon}(\Omega_{\mathbb{T}}(M\times \mathbb{T})).
$$
To this end we have simply to employ the properties of the generalized trace. Indeed, for
$n\geq 2$ we can write 
\begin{align*}
 \left\langle {\rm Tr}_n \left[ 1 \otimes \mathcal \omega_g^{\otimes n} \right]\right\rangle=&\left\langle {\rm Tr}_n \left[ 1 \otimes \omega_g \otimes \omega_g \otimes \mathcal \omega_g^{\otimes (n-2)} \right]\right\rangle\cr
=&-\left\langle {\rm Tr}_n \left[ 1 \otimes dg^{-1} \otimes d g \otimes \mathcal \omega_g^{\otimes (n-2)} \right]\right\rangle \cr
=&-\left\langle {\rm Tr}_n \left[ 1 \otimes dg^{-1} \otimes d g \otimes \mathcal \omega_g^{\otimes (n-2)}\right]\right\rangle\cr &-\left\langle {\rm Tr}_{n-1} \left[ g^{-1} \otimes d g \otimes \mathcal \omega_g^{\otimes (n-2)}\right] \right\rangle \cr
&+\left\langle {\rm Tr}_{n-1} \left[ 1 \otimes g^{-1} d g \otimes \mathcal \omega_g^{\otimes (n-2)} \right]\right\rangle,
\end{align*}
where the last two terms cancel each other because of the trace property, which is precisely of the form (\ref{titty}) for $f=g^{-1}$. Similarly, for $n=1$ it is sufficient to notice that
\begin{align*}
\left\langle  {\rm Tr}_1 \left[ 1\otimes \omega_g \right]\right\rangle= \left\langle {\rm Tr}_1 \left[ g^{-1}\otimes dg \right]\right\rangle,
\end{align*}
which is of the form (\ref{titty0}) with $f=g^{-1}$, completing the proof of $b\mathrm{Ch}^-(g)\in  \mathscr{D}_{\epsilon}(\Omega_{\mathbb{T}}(M\times \mathbb{T}))$.\\
b) It suffices to prove the following two facts:\\
i) If $g,h\in C^{\infty}(M,U(l\times l;\IC))$, then one has $\mathrm{Ch}^-(g\oplus h)=\mathrm{Ch}^-(g)+\mathrm{Ch}^-(h)$.\\ 
ii) If $g_0,g_1\in C^{\infty}(M,U(l\times l;\IC))$ are connected by a smooth homotopy 
$$
g_{\cdot}\in C^{\infty}(M\times I,U(l\times l;\IC)),
$$
 then one has
$$
\mathrm{Ch}^-(g_1) - \mathrm{Ch}^-(g_0) =(b+B)w\quad\text{ in $\mathscr{N}_{\epsilon}(\Omega_{\mathbb{T}}(M\times \mathbb{T}))$}
$$
for some $w\in  \mathscr{C}_{\epsilon}(\Omega_{\mathbb{T}}(M\times \mathbb{T}))$.\\
Here, property i) is an immediate consequence of the properties of the generalized trace ${\mathrm Tr}_n$ using the block diagonal form of $g\oplus h$.\\
To see ii), for any $t\in I$, we define the embedding
\begin{align*}
 j_t:M \hookrightarrow M\times I, \quad x\longmapsto (x,t),
\end{align*}
and $w=(w_0,w_1,\dots)\in  \mathscr{C}_{\epsilon}(\Omega_{\mathbb{T}}(M\times \mathbb{T}))$ by setting
{\footnotesize
\begin{align*}
w_n : =&
-{\rm Tr}_n \left[ \int_0^1 \int_0^1  1 \otimes \sum_{k=1}^n  \sum_{l=0}^{k-2} j^*_t  \left(\mathcal A^s(g_{\cdot})^{\otimes l} \otimes  \iota_{\partial_I} \mathcal A^s(g_{\cdot}) \otimes \mathcal A^s(g_{\cdot})^{\otimes (k-l-2)} \right.\right.\cr
&\left.\left. \qquad\qquad\qquad\qquad \phantom{\int_0^1 \int_0^1  1 \otimes wwwwwwwwwww} \otimes \mathcal B^s(g_{\cdot}) \otimes  {\mathcal A}^s(g_{\cdot})^{\otimes (n-k)} \right)  \ ds \ dt \right]\cr
&+{\rm Tr}_n \left[\int_0^1  \int_0^1 1 \otimes \sum_{k=1}^n  \sum_{l=0}^{n-k-1} j^*_t \left(\mathcal A^s(g_{\cdot})^{\otimes (k-1)} \otimes \mathcal B^s(g_{\cdot}) \otimes \mathcal A^s(g_{\cdot})^{\otimes l}  \otimes  \iota_{\partial_I} \mathcal A^s(g_{\cdot})
\right. \right. \cr 
&\left.\left. \qquad\qquad\qquad\qquad \phantom{\int_0^1 \int_0^1  1 \otimes wwwwwwwwwww}\otimes  {\mathcal A}^s(g_{\cdot})^{\otimes (n-k-l-1)} \right) \ ds \ dt \right]\cr 
&-{\rm Tr}_n \left[\int_0^1 \int_0^1 1 \otimes \sum_{k=1}^n   j^*_t \left( \mathcal A^s(g_{\cdot})^{\otimes (k-1)} \otimes   \iota_{\partial_I} \mathcal B^s(g_{\cdot}) \otimes \mathcal A^s(g_{\cdot})^{\otimes (n-k)} \right) \ ds \ dt \right].
\end{align*}
}
The $\mathscr{C}_{\epsilon}$ growth conditions are easily checked for $w$. Then again it is clear that $B w \in \mathscr{D}_{\epsilon}(\Omega_{\mathbb{T}}(M\times \mathbb{T}))$. On the other hand, by using the identity
\begin{align*}
 d j^*_t   \iota_{\partial_I} \mathcal A^s(g_{\cdot}) =-j^*_t   \iota_{\partial_I} d \mathcal A^s(g_{\cdot}) +\frac \partial{\partial t}  j^*_t \mathcal A^s(g_{\cdot}),
\end{align*}
and similarly for $\mathcal B^s$, and the same computations as in part a) we get, as elements in the Chen normalized complex,
\begin{align*}
(bw + B w)_n&= (bw)_n = (b \left\langle w_n\right\rangle )_n +(b  \left\langle w_{n+1}\right\rangle )_n= \left(\langle \int_0^1 \frac d{d t}  j^*_t \mathrm{Ch}^-(g_.) \right)_n \cr
&= \mathrm{Ch}^-_n(g_1) - \mathrm{Ch}^-_n(g_0) .
\end{align*}
This completes the proof.
\end{proof}

If $M$ has no boundary (so that $LM$ is a well-defined Fr\'{e}chet manifold), in view of $(d+P) \rho= \rho (b+B)$, we immediately get:

\begin{Corollary}\label{sway} Assume $M$ is a compact manifold without boundary. Then for all $g\in C^{\infty}(M,U(l\times l;\IC))$ one has $(d+P) \rho(\mathrm{Ch}^-(g)) =0$ in $\mathscr{N}_{\epsilon}(\Omega_{\mathbb{T}}(M\times \mathbb{T}))$, in particular, $ \rho(\mathrm{Ch}^-(g))$ induces a homology class in $\widehat{\mathsf{H}}^-_\mathbb{T}(LM)$.
\end{Corollary}

\begin{Remark}\label{remark even}
There is an even version of $\mathrm{Ch}^-(g)$ given as follows: If $N$ is a manifold and $d+C$ is a connection on a trivial vector bundle over $N$, then with $R_C$ the curvature of the connection $1$-form $C$ one defines
$$
\mathrm{Ch}^+(C)=(\mathrm{Ch}^+_0(C),\mathrm{Ch}^+_1(C),\dots)\in \mathscr{C}^{+}_{\epsilon}(\Omega_{\mathbb{T}}(N\times \mathbb{T}))
$$
by
$$
\mathrm{Ch}^+_n(C):={\rm Tr}_n \left[ 1\otimes (C -\vartheta_\mathbb{T}\wedge R_C )^{\otimes n} \right],
$$
which by an analogous calculation as in the proof of Theorem \ref{main} is seen to satisfy 
$$
(b+B)\mathrm{Ch}^+(C)=0\quad\text{ in $\mathscr{N}_{\epsilon}(\Omega_{\mathbb{T}}(N\times \mathbb{T}))$.}.
$$
Then, there holds an even/odd periodicity, that is, one can obtain $\mathrm{Ch}^-(g)$ from its even variant by a fiber integration: indeed, by varying $s\in I$ in 
$$
A^s(g) \in \Omega_\mathbb{T}(M , \mathrm{Mat}(l\times l;\IC))
$$
we get a form
$$
A(g) \in \Omega_\mathbb{T}(M\times  I , \mathrm{Mat}(l\times l;\IC))
$$
and can consider the fibration
$$
\pi: M\times I\longrightarrow M.
$$
Then, for the connection $d+\tilde{A}_g$ on the trivial vector bundle over $M\times I$, where $\tilde{A}_g:=\pi^* A_g$, one has, using the definitions of $\mathcal A^s(g)$ and $\mathcal B^s(g)$ that
$$
{\rm Ch}^-(g)=\int_I \iota_{\partial_I} {\rm Ch}^+(\tilde{A}_g)=\pi_*{\rm Ch}^+(\tilde{A}_g) ,
$$
the integration along the fibers of $\pi$.
\end{Remark}

The \emph{odd Chern character} $\mathrm{ch}^-(g)\in\Omega^{-}(M)$ is the closed odd differential form defined by
\begin{align}\label{daaqd}
\mathrm{ch}^-(g):= \mathrm{Tr}\left[ \sum_{j=0}^{\infty}\frac{(-1)^jj!}{(2j+1)!} (g ^{-1}  dg)^{\wedge (2j+1)}\right],
\end{align}
and the \emph{odd Bismut-Chern character} is the differential form 
$$
\mathrm{Bch}^{-}(g)=(\mathrm{Bch}^{-}_1(g),\mathrm{Bch}^{-}_3(g),\dots)\in \widehat{\Omega}^{-}(LM)
$$
defined by
{\footnotesize
\begin{align*}
 \mathrm{Bch}^{-}_{2n-1}(g)=&{\rm Tr} \left[   \int_0^1\int_{\{0\leq t_1\leq \dots t_n\leq 1\}} \sum_{j=1}^n  \bigwedge_{i=1}^{j-1} \trasp^{s}_{t_i}(g) R^s_g(t_i)\bigwedge \trasp^{s}_{t_j}(g)  \dot{A}^s_g(t_j) \right. \cr
& \left. \phantom{mangialatuttalapastasciutta} 
 \bigwedge_{l=j+1}^n\trasp^{s }_{t_l }(g) R^s_g(t_l)   \trasp^s_1 (g) dt_1\cdots dt_n ds \right],
\end{align*}
}
where 
$$
\dot{A}^s_g= \frac d{ds} A^s_g  =\omega_g\in \Omega^1(M,\mathrm{Mat}(l\times l;\IC)),
$$
 and where $\trasp^{s}_{\cdot}(g)$ denotes the parallel transport with respect to the connection $d+s\omega_g$ on the trivial vector buncle over $M$.

\begin{Theorem}\label{main2} Assume $M$ is a compact Riemannian manifold, possibly with boundary, and let $g\in C^{\infty}(M,U(l\times l;\IC))$. Then one has $\rho(\mathrm{Ch}^{-}(g))|_M=\mathrm{ch}^{-}(g)$, and if $M$ has no boundary then $\mathrm{Bch}^{-}(g)=\rho(\mathrm{Ch}^{-}(g))$.
\end{Theorem}

Note that in view of Corollary \ref{sway}, Theorem \ref{main2} provides a new proof of 
$$
(d+P)\mathrm{Bch}^{-}(g)=0
$$
We refer the reader to \cite{wilson} for a variant of this result.

\begin{proof}[Proof of Theorem \ref{main2}] The formula $\rho(\mathrm{Ch}^{-}(g))|_M=\mathrm{ch}^{-}(g)$ is a simple consequence of the definitions, once one has noticed the formula
$$
\rho\left\langle (\alpha_0+\vartheta_\mathbb{T}\wedge \beta_0 ) \otimes \dots \otimes(\alpha_{n}+\vartheta_\mathbb{T}\wedge\beta_{n}) \right\rangle|_M= \alpha_0\wedge \cdots \wedge\alpha_n.
$$
In order to see $\mathrm{Bch}^{-}(g)=\rho(g)$, given $t,s\in I$ define 
$$
V^{s}(g,t)\in  \widehat{\Omega}^{-}(LM, \mathrm{Mat}(l\times l;\IC))
$$
by
  \begin{align*}
 V^{s}_{2n+1}(g,t)=    \int_{\{0\leq t_1\leq \dots t_{n+1}\leq t\}} &\sum_{j=1}^{n+1}  \bigwedge_{i=1}^{j-1} \trasp^{s}_{t_i}(g) R^s_g(t_i)\bigwedge \trasp^{s}_{t_j}(g)  \dot{A}^s_g(t_j)\\
& \times \bigwedge_{l=j+1}^{n+1}\trasp^{s }_{t_l }(g) R^s_g(t_l)   \trasp^s_1(g) dt_1\cdots d t_{n+1}  ,
\end{align*}
and the differential form 
$$
W^{s}(g,t)\in   \widehat{\Omega}^{-}(LM, \mathrm{Mat}(l\times l;\IC))
$$
by
\begin{align*}
& W^{s}_{2n+1}(g,t)=    \sum_{k=n+1}^\infty  \ \      \sum_{r,j_1,\cdots,j_{n}=1, \text{pairwise distinct}}^k\\
 &\times  \int_{\{0\leq t_1\leq \dots t_k\leq  t \}}   \iota A^s_g(t_1) \cdots R^s_g (t_{j_1}) \cdots  \dot A^s_g(t_r) \cdots  R^s_g  (t_{j_n})\cdots \iota  A^s_g(t_k) dt_1\cdots dt_k.
 \end{align*}


 Then obviously one has 
$$
 \mathrm{Bch}^{-}(g)= \mathrm{Tr}  \left[ \int^1_0  V^{s}(g,t)|_{t=1} ds\right]
$$
and it is easily checked from the definitions that 
$$
 \rho (\mathrm{Ch}^-(g) )=\mathrm{Tr} \left[ \int^1_0 W^{s}(g,t)|_{t=1} ds\right].
$$

Thus it suffices to show that $W^{s}(g,t)=V^{s}(g,t)$ for all $t,s\in I$. To see this, the essential idea is to consider for every $t,s\in I$ the even form
$$
X^{s}(g,t)=(X^{s}_0(g,t),X^{s}_2(g,t),\dots)\in  \widehat{\Omega}^{+}(LM, \mathrm{Mat}(l\times l;\IC)) 
$$
which is defined by  
\begin{align*}
& X^{s}_{0}(g,t)=\pa^s_{t}(g),\\
& \frac d{dt} X^{s}_{2n}(g,t)=X^{s}_{2n}(g,t)\iota A^s_g(t)+X^{s}_{2n-2}(g,t) R^s_g(t),\\
&X^{s}_{2n}(g,t)|_{t=0}=0\quad\text{ for all $n\geq 1$,}\\
\end{align*}
and the odd form
$$
Y^{s}(g,t)=(Y^{s}_1(g,t),Y^{s}_3(g,t),\dots)\in  \Omega^{-}(LM, \mathrm{Mat}(l\times l;\IC)) 
$$
which is defined by
\begin{align*}
& \frac {d}{dt} Y^{s}_{1}(g,t)=Y^{s}_{1}(g,t)\iota A_g^s(t)+X^{s}_{0}(g,t) \dot{A}^s_g(t),\\
& \frac {d}{dt} Y^{s}_{2n+1}(g,t)=Y^{s}_{2n+1}(g,t)\iota A^s_g(t)+Y^{s}_{2n-1}(g,t) R^s_g(t) + X^{s}_{2n}(g,t) \dot{A}^s_g(t) \quad \forall n\geq 1,\\
&Y^{s}_{2n+1}(g,t)|_{t=0}=0\quad\text{ for all $n$}.
\end{align*}
Noting that the sum that defines $W^{s}_{2n+1}(g,t)$ converges uniformly in $t$ so that one can interchange $d/dt$ with $\sum^{\infty}_{k=n+1}$, it is now easily checked that both $t\mapsto W^s(g,t)$ and $t\mapsto V^s(g,t)$ solve the IVP's which define $Y^{s}(g,t)$, so that 
$$
V^{s} (g,t)=W^{s} (g,t)=Y^{s}(g,t)\quad\text{ for all $t,s\in I$},
$$
as was claimed.

\end{proof}
 
\begin{Remark}
If $N$ is a compact manifold without boundary and given a connection $d+C$ over a trivial vector bundle over $N$, the \emph{even Bismut-Chern character} is the differential form 
$$
\mathrm{Bch}^{+}(C)=(\mathrm{Bch}^{+}_0(C),\mathrm{Bch}^{+}_2(C),\dots)\in \widehat{\Omega}^{+}(LN)
$$
defined by
\begin{align*}
 \mathrm{Bch}^{+}_{2n}(C)=&{\rm Tr} \left[ \int_{\{0\leq t_1\leq \dots t_n\leq 1\}} \bigwedge_{i=1}^{n} \trasp^C_{t_i} R_{C}(t_i)   \trasp^C_1 dt_1\cdots dt_n   \right],
\end{align*}
where $R_C$ is again the curvature of $d+C$ and $\trasp^C_{\cdot}$ is the parallel transport with respect to $d+C$. Then one has another even/odd periodicity as in Remark \ref{remark even}: we can consider $A^s_g$ as defining a connection 
$1$-form $\tilde{A}_g$ over a trivial vector bundle over $M\times I$. However, since $M\times I$ is a manifold with boundary, it is convenient to embed it in a larger manifold, say 
$$
\chi: M\times I \hookrightarrow M\times J
$$ 
where $J=(-1,2)$. Therefore, we extend $A^s_g$ to $s\in J$, consider it as defining a connection $1$-form $\tilde{A}_g$ over a trivial vector bundle over $M\times J$.\\
The corresponding curvature 
$$
R_{\tilde{A}_g}\in  \Omega^2(M\times J,\mathrm{Mat}(l\times l;\IC))
$$
is given by varying $s\in J$ in 
$$
R^s_g+ds \wedge   \dot A^s_g \in \Omega^2(M,\mathrm{Mat}(l\times l;\IC)).
$$
Since $\iota_{\partial J} R_{\tilde{A}_g}=\dot A^s_g$, after restricting to loops fibering over $J$, we immediately get that under integration along the fibers of  
$$
\pi:M\times I\longrightarrow M,
$$
one has
$$
\mathrm{Bch}^{-}_{2n-1}(g)= \int_I \chi^* \iota_{\partial_J} \mathrm{Bch}^{+}_{2n}(\tilde{A}_g) =\pi_* \chi^*\mathrm{Bch}^{+}_{2n}(\tilde{A}_g) .
$$
\end{Remark}


\end{document}